\documentclass{gtmon_a}
\pdfoutput=1
\usepackage[all]{xy}

\proceedingstitle{Groups, homotopy and configuration spaces (Tokyo
  2005)}
\conferencestart{5 July 2005}
\conferenceend{11 July 2005}
\conferencename{Groups, homotopy and configuration spaces,
                in honour of Fred Cohen's 60th birthday}
\conferencelocation{University of Tokyo, Japan}

\editor{Norio Iwase}
\givenname{Norio}
\surname{Iwase}

\editor{Toshitake Kohno}
\givenname{Toshitake}
\surname{Kohno}

\editor{Ran Levi}
\givenname{Ran}
\surname{Levi}

\editor{Dai Tamaki}
\givenname{Dai}
\surname{Tamaki}

\editor{Jie Wu}
\givenname{Jie}
\surname{Wu}

\title[Symmetric products and configuration spaces]{Symmetric products, 
duality and homological\\dimension of configuration spaces}
\author{Sadok Kallel}
\givenname{Sadok}
\surname{Kallel}
\address{Universit\'e des Sciences et Technologies de Lille\\
  Laboratoire Painlev\'e, U.F.R de Math\'ematiques\\\newline
 59655 Villeneuve d'Ascq\\France}
\email{sadok.kallel@math.univ-lille1.fr}

\dedicatory{To Fred Cohen on his 60th birthday}

\volumenumber{13}
\issuenumber{}
\publicationyear{2008}
\papernumber{23}
\startpage{499}
\endpage{527}

\doi{}
\MR{}
\Zbl{}

\arxivreference{}


\received{1 July 2006}
\revised{31 May 2008}
\accepted{16 June 2008}
\published{26 July 2008}
\publishedonline{26 July 2008}
\proposed{}
\seconded{}
\corresponding{}
\version{}

\makeatletter
\def\cnewtheorem#1[#2]#3{\newtheorem{#1}{#3}[section]
\expandafter\let\csname c@#1\endcsname\c@thm}

\let\xysavmatrix\xymatrix
\def\xymatrix{\disablesubscriptcorrection\xysavmatrix}
\AtBeginDocument{\let\bar\wbar\let\tilde\wtilde\let\hat\what}
\makeop{Map}
\makeop{Hom}
\makeop{conn}
\makeop{min}
\makeop{im}

\newtheorem{thm}{Theorem}[section]  
\cnewtheorem{lem}[thm]{Lemma}        
\cnewtheorem{prop}[thm]{Proposition}
\cnewtheorem{cor}[thm]{Corollary}
\theoremstyle{definition}
  
\newtheorem*{rem}{Remark}           
\cnewtheorem{exam}[thm]{Example}

\makeatother  

\def\la#1{\hbox to #1pc{\leftarrowfill}}
\def\ra#1{\hbox to #1pc{\rightarrowfill}}
\def\fract#1#2{\raise4pt\hbox{$ #1 \atop #2 $}}
\def\decdnar#1{\phantom{\hbox{$\scriptstyle{#1}$}}
\left\downarrow\vbox{\vskip10pt\hbox{$\scriptstyle{#1}$}}\right.}

\def\lrar{{\ra{1.5}}}

\def\tensor{\otimes}
\def\sp#1{{SP}^{#1}}
\def\bsp#1{\overline{{SP}}^{#1}}
\def\tsp#1{{{\mathbb S} P}^{#1}}
\def\spy{{SP}^{\infty}}

\def\map#1{\Map_{#1}}
\def\bmap#1{\Map^*_{#1}}
\def\tp#1{\hbox{TP}^{#1}}

\def\cohdim{\hbox{cohdim}}
\def\bbz{{\mathbb Z}}
\def\bbf{{\mathbb F}}
\def\bbp{{\mathbb P}}
\def\bbr{{\mathbb R}}
\def\bbn{{\mathbb N}}
\def\bbc{{\mathbb C}}
\def\bbr{{\mathbb R}}

\begin{document}

\begin{htmlabstract}
We discuss various aspects of &ldquo;braid spaces&rdquo; or
configuration spaces of unordered points on manifolds.  First we
describe how the homology of these spaces is affected by puncturing
the underlying manifold, hence extending some results of Fred Cohen,
Goryunov and Napolitano.  Next we obtain a precise bound for the
cohomological dimension of braid spaces.  This is related to some
sharp and useful connectivity bounds that we establish for the reduced
symmetric products of any simplicial complex.  Our methods are
geometric and exploit a dual version of configuration spaces given in
terms of truncated symmetric products.  We finally refine and then
apply a theorem of McDuff on the homological connectivity of a map
from braid spaces to some spaces of &ldquo;vector fields&rdquo;.
\end{htmlabstract}

\begin{abstract} 
We discuss various aspects of ``braid spaces'' or configuration spaces
of unordered points on manifolds.  First we describe how the homology
of these spaces is affected by puncturing the underlying manifold,
hence extending some results of Fred Cohen, Goryunov and Napolitano.
Next we obtain a precise bound for the cohomological dimension of
braid spaces.  This is related to some sharp and useful connectivity
bounds that we establish for the reduced symmetric products of any
simplicial complex.  Our methods are geometric and exploit a dual
version of configuration spaces given in terms of truncated symmetric
products.  We finally refine and then apply a theorem of McDuff on the
homological connectivity of a map from braid spaces to some spaces of
``vector fields''.
\end{abstract}

\maketitle


\section{Introduction}

Braid spaces or configuration spaces of \textit{unordered pairwise
distinct} points on manifolds have important applications to a number
of areas of mathematics and physics. They were of crucial use in the
seventies in the work of Arnold on singularities and then later in the
eighties in work of Atiyah and Jones on instanton spaces in gauge
theory. In the nineties they entered in many works on the homological
stability of holomorphic mapping spaces. No more important perhaps had
been their use than in stable homotopy theory in the sixties and early
seventies through the work of Milgram, May, Segal and Fred Cohen who
worked out the precise connection with loop space theory. This work
has led in particular to the proof of Nishida's nilpotence theorem and
to Mahowald's infinite family in the stable homotopy groups of spheres
to name a few.

Given a space $M$, define $B(M,n)$ to be the space of finite subsets
of $M$ of cardinality $n$.  This is usually referred to as the $n^{\rm th}$
``braid space'' of $M$ and in the literature it is often denoted by
$C_n(M)$ (Atiyah and Jones \cite{atiyah}, B{\"o}digheimer, Cohen and
Taylor \cite{bct}, Cohen \cite{cohen}). Its fundamental group written
$Br_n(M)$ is the ``braid group'' of $M$. The object of this paper is
to study the homology of braid spaces and the main approach we adopt
is that of duality with the symmetric products. In so doing we take
the opportunity to refine and elaborate on some classical material.
Next is a brief content summary.

\fullref{braids} describes the homotopy type of braid spaces of some
familiar spaces and discusses orientation issues.  \fullref{tp}
introduces truncated products, as in B{\"o}digheimer, Cohen and
Milgram \cite{bcm} and Milgram and L{\"o}ffler \cite{lm}, states the
duality with braid spaces and then proves our first main result on the
cohomological dimension of braid spaces. \fullref{punctured} uses
truncated product constructions to split in an elementary fashion the
homology of braid spaces for punctured manifolds.  In \fullref{bounds}
we prove our sharp connectivity result for \textit{reduced} symmetric
products of CW complexes which seems to be new and a significant
improvement on work of Nakaoka and Welcher \cite{welcher}.  In
\fullref{spec} we make the link between the homology of symmetric and
truncated products by discussing a spectral sequence introduced by
B\"odigheimer, Cohen and Milgram and exploited by them
to study ``braid homology'' $H_*(B(M,n))$. Finally \fullref{stability}
completes a left out piece from McDuff and Segal's work on
configuration spaces \cite{dusa}.  In that paper, $H_*(B(M,n))$, for
closed manifolds $M$, is compared to the homology of some spaces of
``compactly supported vector fields'' on $M$ and the main theorem there
states that these homologies are isomorphic up to a range that
increases with $n$. We make this range more explicit and use it for
example to determine the abelianization of the braid groups of a
closed Riemann surface. A final appendix collects some homotopy
theoretic properties of section spaces that we use throughout.

Below are precise statements of our main results which we have divided
up into three main parts.  Unless explicitly stated, all spaces are
assumed to be connected. The $n^{\rm th}$ symmetric group 
is written $\mathfrak{S}_n$.

\subsection{Connectivity and cohomological dimension}
For $M$ a manifold, we write $H^*(M,\pm \bbz )$ for the cohomology of $M$
with coefficients in the orientation sheaf $\pm\bbz$; in other words
$H^*(M,\pm\bbz)$ is the homology of\break $\Hom_{\bbz[\pi_1(X)]}(C_*(\tilde
M), \bbz)$, where $C_*(\tilde M)$ is the singular chain complex of the
universal cover $\tilde M$ of $M$, and where the action of (the class
of) a loop on the integers $\bbz$ is multiplication by $\pm 1$
according to whether this loop preserves or reverses orientation.
Similarly one defines $H_*(M,\pm\bbz ):= 
H_*(C_*(\tilde M)\tensor_{\bbz[\pi_1(x)]}\bbz)$.

\begin{rem}\label{twisted} (see \fullref{folklore})\qua When $M$ is simply connected
  and $\dim M:=d > 2$, $\pi_1(B(M,k))=\mathfrak{S}_k$ and $\tilde B(M,k) =
  F(M,k)\subset M^k$ is the subspace of $k$ \textit{ordered} pairwise distinct
  points in $M$ (\fullref{braids}). It follows that
  $H^*(B(M,k);\pm\bbz)$ is the homology of the chain complex
  $\Hom_{\bbz[\mathfrak{S}_k]}(C_*(F(M,k),\bbz)$ where $\mathfrak{S}_k$ acts on
  $\bbz$ via $\sigma (1) = (-1)^{sg( \sigma)\cdot d }$ and $sg (\sigma
  )$ is the sign of the permutation $\sigma\in\mathfrak{S}_k$. \end{rem}

We denote by $\cohdim_{\pm\bbz}(M)$ (cohomological dimension) the smallest
integer with the property that
$$H^i(M;\pm\bbz ) = 0\ ,\ \ \forall i > \cohdim_{\pm\bbz }(M)\ .$$
If $M$ is orientable, then $H^*(M,\pm\bbz ) = H^*(M,\bbz )$ and
$\cohdim_{\pm\bbz}(M) =$\break $\cohdim (M)$, the
cohomological dimension of $M$.

A space $X$ is $r$--connected if $\pi_i(X)=0$ for $0\leq i\leq r$. The
connectivity of $X$; $\conn(X)$, is the largest integer with such a
property.  This connectivity is infinite if $X$ is contractible.  The
following is our first main result

\begin{thm}\label{main3} Let $M$ be a compact manifold of dimension $d\geq 1$,
  with boundary $\partial M$, and let $U\subset M$ be a closed subset such
  that $U\cap\partial M= \emptyset$ and $M-U$ connected. 
We denote by $r$ the connectivity of $M$
  if $U\cup\partial M=\emptyset$, or the connectivity of the quotient
  $M/U\cup\partial M$ if $U\cup\partial M\neq \emptyset$. We assume $0\leq
  r<\infty$ and $k\geq 2$. Then
$$
\cohdim_{\pm\bbz}(B(M-U,k)) \leq
\begin{cases} (d-1)k-r+1, &\hbox{if}\ U\cup\partial M=\emptyset,\\
(d-1)k-r, &\hbox{if}\ U\cup\partial M\neq \emptyset.
\end{cases}
$$ When $M$ is even dimensional orientable, then replace
$\cohdim_{\pm\bbz}$ by $\cohdim$.
\end{thm}

\begin{rem}\label{numbered}
We check this theorem against some known examples:
\begin{enumerate}
\item $B(S^d-\{p\},2)=B(\bbr^d,2)\simeq\bbr P^{d-1}$
(see \fullref{tp}) and
$\cohdim_{\pm\bbz}(B(\bbr^d,2))$
$= 2(d-1)-r = d-1 =\cohdim_{\pm\bbz}(\bbr P^{d-1})$ indeed, where
$r=d-1= \conn(S^d)$.
\item
$B(S^d,2)\simeq\bbr P^d$ (see \fullref{tp}) and
$\cohdim_{\pm\bbz} (B(S^d,2)) = d$\ in agreement with our formula.
\item
It is known that for odd primes $p$ and $d\geq 2$,
$H^{(d-1)(p-1)}(B(\bbr^d,p);\bbf_p)$ is non-trivial and an isomorphic
image of $H^{(d-1)(p-1)}(\mathfrak{S}_p;\bbf_p)$ (Ossa \cite{ossa} and
Vassiliev \cite{vassiliev}).  Our result states that, at least for
even $d$, no higher homology can occur.  The cohomological dimension
of $B(\bbr^d,k)$ when using $\bbf_2$ coefficients is known to be
$(k-\alpha (k))\cdot (d-1)$ where $\alpha (k)$ is the number of 1's in
the dyadic decomposition of $k$ (see Roth \cite{frido}).  In the case
$d=2$, $B(\bbr^2,k)$ is the classifying space of Artin braid group
$B_k:= Br_k(\bbr^2)$ and is homotopy equivalent to a
$(k-1)$--dimensional CW complex so that $\cohdim (B(\bbr^2,k))\leq k-1$
in agreement with our calculation.
\end{enumerate}
\end{rem}

\begin{rem} The theorem applies to when $M=S^1$ and $U$ is either empty
or a single point. In that case $M-U\cong S^1,\bbr$. But one knows that
for $k\geq 1$, $B(S^1,k)\simeq S^1$ (\fullref{cns1}) 
and $B(\bbr,k)$ is contractible.
\end{rem}

\begin{cor}\label{twocomplexes} Let $S$ be a Riemann surface and
  $Q\subset S$ a finite subset.  Then $H^i(B(S-Q,k)) = 0$ if $i\geq k+1$ and
  $Q\cup\partial S\neq\emptyset$ ; or if $i> k+1$ and $Q\cup\partial
  S=\emptyset$.  \end{cor}

This corollary gives an extension of the ``finiteness'' result of
Napolitano \cite{nap1}.  When $S$ is an open surface, then $B(S,k)$ is a Stein
variety and hence its homology vanishes above the complex dimension;
ie, $H_i(B(S,k)) = 0$ for $i> k$. This also agrees with the above computed bounds.

The proof of \fullref{main3} relies on a useful connectivity result of
Nakaoka (\fullref{nakak}). We also use this result to produce sharp
connectivity bounds for the \textit{reduced} symmetric products
\fullref{bounds}.  Recall that $\sp{n}(X)$, the $n^{\rm th}$ symmetric product of
$X$, is the quotient of $X^n$ by the permutation action of the symmetric group
$\mathfrak{S}_n$ so that $B(X,n)\subset\sp{n}(X)$ is the subset of
configurations of distinct points. We always assume $X$ is based so there is
an embedding $\sp{n-1}(X)\hookrightarrow\sp{n}(X)$ given by adjoining the
basepoint, with cofiber $\bsp{n}(X)$ the ``$n^{\rm th}$ reduced symmetric'' product of
$X$. The following result expresses the connectivity of $\bsp{n}X$ in terms of
the connectivity of~$X$.

\begin{thm}\label{connectivity} Suppose $X$ is a based $r$--connected simplicial
  complex with $r\geq 1$. Then $\bsp{n}(X)$ is $(2n+r-2)$--connected. 
\end{thm}

In particular the embedding $\sp{n-1}(X)\lrar\sp{n}(X)$ induces
homology isomorphisms in degrees up to $(2n+r-3)$.  The proof of this
theorem is totally inspired from Kallel and Karoui \cite{kk} where
similar connectivity results are stated, and it uses the fact that the
homology of symmetric products only depends on the homology of the
underlying complex (Dold \cite{dold}).  Note that the bound $2n+r-2$
is sharp as is illustrated by the case $X=S^2$, $r=1$ and
$\bsp{n}(S^2)=S^{2n}$.  A slightly weaker connectivity bound than ours
can be found in Welcher \cite[Corollary 4.9]{welcher}.

Note that \fullref{connectivity} is stated for simply connected
spaces.  To get connectivity results for reduced symmetric products of
a compact Riemann surface for example we use geometric input from
Kallel and Salvatore \cite{ks2}. This applies to any two dimensional
complex.

\begin{prop}\label{conntwo} Let $X = \bigvee^wS^1\cup
(D^2_1\cup\cdots\cup D^2_r)$ be a two dimensional CW complex with one
skeleton a bouquet of $w$ circles.  Then $\bsp{n}X$ is
$(2n-\min(w,n)-1)$--connected.
\end{prop}


\subsection{Puncturing manifolds}
We give generalizations and a proof simplification of results of
Napolitano \cite{nap1,nap2}.  For $S$ a two dimensional topological
surface, $p$ and the $p_i$ points in $S$, it was shown in \cite{nap1}
that, for field coefficients $\bbf$,
\begin{equation}\label{secondsplit}
H^j(B(S -\{p_1,p_2\},n);\bbf )\cong
\bigoplus_{t=0}^nH^{j-t}(B(S -\{p\},n-t);\bbf )\ .
\end{equation}
Here and throughout $H^* = 0$ when $*<0$ and $B(X,0)$ is basepoint.
When $S$ is a closed orientable surface and
$\bbf=\bbf_2$, \cite{nap1} establishes furthermore a splitting:
\begin{equation}\label{firstsplit}
H^j(B(S,n);\bbf_2)\cong
H^{j}(B(S -\{p\},n);\bbf_2)\oplus H^{j-2}(B(S -\{p\},n-1);\bbf_2)
\end{equation}
Similar splittings occur in 
Cohen \cite{cohen2} and Gorjunov \cite{goryunov}. These splittings as we show extend
to any closed topological manifold $M$ and to any number of
punctures. If $V$ is a vector space, write
$V^{\oplus k}:= V\oplus\cdots\oplus V$ ($k$--times). Given
positive integers $r$ and $s$, we write $p(r,s)$ the number of ways we can
partition $s$ into a sum of $r$ \textit{ordered} positive (or null) integers.
For instance $p(1,s)=1$, $p(2,s)=s+1$ and $p(r,1)=r$.

\begin{thm}\label{main} Let $M$ be a closed connected manifold of
  dimension $d$ and $p\in M$. Then:
\begin{equation}\label{main11}
H^j(B(M,n);\bbf_2)\cong
H^{j}(B(M-\{p\},n);\bbf_2)\oplus H^{j-d}(B(M-\{p\},n-1);\bbf_2)
\end{equation}
If moreover $M$ is oriented and even dimensional, then:
\begin{align}\label{main22}
H^j(B(M-&\{p_1,\cdots, p_k\},n);\bbf )\\
&\cong \bigoplus_{0\leq r\leq n}
H^{j - (n-r)(d-1)}
(B(M-\{p\},r);\bbf )^{\oplus p(k-1,n-r)}\notag
\end{align}
For an arbitrary closed manifold, \eqref{main22}
is still true with $\bbf_2$--coefficients.
\end{thm}

\begin{rem}
As an example we can set $M=S^2, k=2=d$ and obtain the
additive splitting $H^j(B(\bbc^*,n);\bbf )\cong
\bigoplus_{0\leq r\leq n} H^{j - (n-r)} (B(\bbc ,r);\bbf )$\ 
as in \eqref{secondsplit}, where
$\bbc^*$ is the punctured disk (this isomorphism holds integrally
according to \cite{goryunov}).  Note that the left hand side is the
homology of the hyperplane arrangement of `Coxeter type'' $B_n$;
that is $B(\bbc^*,n)$ is an Eilenberg--MacLane space
$K(Br_n(\bbc^*),1)$ with fundamental group isomorphic to the subgroup of 
Artin's braids $Br_{n+1}(\bbc )$ consisting of those braids which leave
the last strand fixed. It can be checked that the abelianization
of this group for $n\geq 2$ is $\bbz^2$ which is consistent with the calculation
of $H^1$ obtained from the above splitting.
\end{rem}

Napolitano's approach to \eqref{secondsplit} is through spectral
sequence arguments and ``resolution of singularities'' as in Vassiliev
theory.  Our approach relies on a simple geometric manipulation of the
truncated symmetric products as discussed earlier (see \fullref{punctured}). \fullref{main} is a consequence of combining a
Poincar\'e--Lefshetz duality statement, the identification of truncated
products of the circle with real projective space, Mostovoy
\cite{mostovoy}, and a homological splitting result due to Steenrod
(\fullref{tp}).  Note that the splitting in (3) is no longer true
with coefficients other than $\bbf_2$ and is replaced in general by a
long exact sequence (\fullref{longexact}).

\subsection{Homological stability}\label{homstab}
This is the third and last part of the paper.
 For $M$ a closed smooth manifold of dimension
$\dim M = d$, let $\tau^+M$ be the fiberwise one-point
compactification of the tangent bundle $\tau M$ of $M$ with fiber $S^d$.
We write $\Gamma (\tau^+M)$ the space of sections of $\tau^+M$. Note that this
space has a preferred section (given by the points at infinity).
There are now so called ``scanning'' maps for any $k\in\bbn$
(Mcduff \cite{dusa}, B{\"o}digheimer, Cohen and Taylor \cite{bct},
Kallel \cite{quarterly})
\begin{equation}\label{firstscan}
S_k \co  B(M,k)\lrar \Gamma_k (\tau^+M )
\end{equation}
where $\Gamma_k(\tau^+M)$ is the component of degree $k$ sections (see
\fullref{scan}).  In important work, McDuff shows that $S_k$
induces a homology isomorphism through a range that increases with
$k$. In many special cases, this range needs to be made explicit and
this is what we do next.

We say that a map $f\co  X\rightarrow Y$ is homologically $k$--connected (or a
homology equivalence up to degree $k$) if $f_*$ in homology is an isomorphism
up to and including degree $k$.

\begin{prop}\label{main4}
Let $M$ be a closed manifold of dimension $d\geq 2$ and $k\geq 2$.
Assume the map $+\co B(M-p,k)\lrar B(M-p,k+1)$ which consists of adding a point
near $p\in M$ (see \fullref{stability}) is
homologically $s(k)$--connected. Then
scanning $S_k$ is homologically $s(k-1)$--connected. Moreover
$s(k)\geq [k/2]$ (Arnold). 
\end{prop}

When $k=1$, we give some information about $S_1 \co 
M\lrar\Gamma_1(\tau^+M)$ in \fullref{s1}. Note that $s(k)$ is an
increasing function of $k$.  Arnold's inequality $s(k)\geq [k/2]$ is
proven by Segal in \cite{segal1}.  This bound is far from being
optimal in some cases since for instance, for $M$ a compact Riemann
surface, $s(k) = k-1$ (Kallel and Salvatore \cite{ks}).  Note that the
actual connectivity of the map $+\co B(M-p,k)\lrar B(M-p,k+1)$ is often
$0$ since if $\dim M > 2$, this map is never trivial on $\pi_1$ (see
\fullref{folklore}).

The utility of \fullref{main4} is that in some particular cases, knowledge
of the homology of braid spaces in a certain range informs on the homology of
some mapping spaces. Here's an interesting application to computing the
abelianization of the braid group of a surface (this was an open problem for
some time).

\begin{cor} \label{cor1} For $S$ a compact Riemann surface of genus $g\geq 1$,
and $k\geq 3$, we have the isomorphism:
$H_1(B(S,k);\bbz ) = \bbz_2\oplus\bbz^{2g}$.
\end{cor}

\begin{proof} $\tau^+S$ is trivial since $S$ is stably parallelizable
  and $\Gamma (\tau^+S) \simeq \Map (S, S^2)$. Suppose $S$ has odd genus,
  then $S_k \co  H_1(B(S,k))\lrar H_1(\Map_k(S,S^2))$ is degree preserving
(where degree is $k$)
  and according to \fullref{main4} it is an isomorphism when
  $k\geq 3$ using the bound provided by Arnold.
  But $\pi:=\pi_1(\Map_k(S,S^2))$ was
  computed in \cite{contemp} and it is some extension
  $$0\lrar\bbz_{2|k|}\lrar\pi\lrar \bbz^{2g}\lrar 0$$
  with a generator $\tau$ and torsion free generators $e_1,\ldots, e_{2g}$
  with non-zero commutators $[e_i, e_{g+i}] = \tau^2$ and with $\tau^{2|k|} = 1$.
  Its abelianization $H_1$ is
  $\bbz^{2g}\oplus\bbz_2$ as desired. When $g$ is even,
  $S_k \co  B(S,k)\lrar \Map_{k-1}(S,S^2)$
  decreases degree by one (see \fullref{sectionspace})
  but the argument and the conclusion are still the same.
\end{proof}

\begin{rem} The above corollary is also a recent calculation of
Bellingeri, Gervais and Guaschi \cite{bellingo} which is more
algebraic in nature and relies on the full presentation of the braid
group $\pi_1(B(S,k))$ for a positive genus Riemann surface
$S$.\end{rem}

\begin{exam} We can also apply \fullref{main4} to the case when $M$ is a
sphere $S^n$.  Write $\Map(S^n,S^n) = \coprod_{k\in\bbz}
\map{k}(S^n,S^n)$ for the space of self-maps of $S^n$;
$\map{k}(S^n,S^n)$ being the component of degree $k$ maps.  Since
$\tau^+S^n$ is trivial there is a homeomorphism $\Gamma (\tau^+S^n
)\cong \Map(S^n,S^n)$. However and as pointed out by Salvatore in
\cite{paolo}, one has to pay extra care about components :
$\Gamma_k(\tau^+S^n)\cong \map{k}(S^n,S^n)$ if $n$ is odd and
$\Gamma_k(\tau^+S^n)\cong \map{k-1}(S^n,S^n)$ if $n$ is even (see
\fullref{sectionspace}).  Let $p(n)=1$ if $n$ is even and $0$ if
$n$ is odd.  Vassiliev \cite{vassiliev} checks that
$H_*(B(\bbr^n,k);\bbf_2)\lrar$ $ H_*(B(\bbr^n,k+1);\bbf_2)$ is an
isomorphism up to degree $k$ and so we get that the map of the $k^{\rm th}$
braid space of the sphere into the higher free loop space
$$B(S^n,k)\lrar \map{k-p(n)}(S^n,S^n)$$ is a mod--$2$ homology
equivalence up to degree $k-1$.  
The homology of $\Map(S^n,S^n)$ is worked out for all field coefficients in
\cite{paolo}.
\end{exam}

\begin{rem} The braid spaces fit into a filtered construction
$$B(M,n)=\co  B^1(M,n)\hookrightarrow B^2(M,n)
\hookrightarrow\cdots\hookrightarrow B^n(M,n):=\sp{n}(M)$$ where
$B^p(M,n)$ for $1\leq p\leq n$ is defined to be the subspace
\begin{equation}\label{spnd}
\{[x_1,\ldots, x_n]\in\sp{n}(M)\ |\ \hbox{no more than
$p$ of the $x_i$'s are equal}\}\ .
\end{equation}
Many of our results can be shown to extend with
straightforward changes to $B^p(M,n)$ and $p\geq 1$ when $M$ is a compact
Riemann surface. Some detailed statements and calculations can be found in 
\cite{ks}.  
\end{rem}

{\bf Acknowledgements}\qua  We are grateful to the referee for his
careful reading of this paper.  We would like to thank Toshitake Kohno,
Katsuhiko Kuribayashi and Dai Tamaki for organizing two most enjoyable
conferences first in Tokyo and then in Matsumoto.  Fridolin Roth, Daniel
Tanr\'e and Stefan Papadima have motivated part of this work with relevant
questions. We finally thank Fridolin and Paolo Salvatore for 
commenting through an early version of this paper.


\section{Basic examples and properties}\label{braids}

As before we write an element of $\sp{n}(X)$ as an unordered $n$--tuple
of points\break $[x_1,\ldots, x_n]$ or sometimes also as an abelian finite
sum $\sum x_i$ with $x_i\in X$.  For a closed manifold $M$,
$\sp{n}(M)$ is again a closed manifold for $n>1$ if and only if $M$ is of
dimension two, Wagner \cite{wagner}.  We define
$$B(M,n) = \{[x_1,\ldots, x_n]\in\sp{n}(M), x_i\neq x_j, i\neq j\}\ .$$
It is convenient as well to define the ``ordered'' $n$--fold configuration space
$F(M,n)= M^n - \Delta_{\rm fat}$ where
\begin{equation}\label{fat}
\Delta_{\rm fat}:= \{(x_1,\ldots, x_n)\in M^n\ |\ x_i=x_j\ \hbox{for some}\
i=j\}
\end{equation}
is the \textit{fat diagonal} in $M^n$.
The configuration space $B(M,n)$ is obtained as the quotient
$F(M,n)/\mathfrak{S}_n$ under the free permutation action of $\mathfrak{S}_n$
\footnote{In the early literature on embedding theory, Feder \cite{feder}, 
$B(M,2)$ was referred to
as the ``reduced symmetric square''.}.
Both $F(M,n)$ and $B(M,n)$ are (open) manifolds of dimension
$nd$, $d=\dim M$.

Next are some of the simplest non-trivial braid spaces one can describe.

\begin{lem}\label{c2sn} $B(S^n,2)$ is an open $n$--disc bundle over $\bbr P^n$.
When $n=1$, this is the open M\"{o}bius band (see \fullref{cns1}). 
\end{lem}

\begin{proof} There is a surjection $\pi \co  B(S^n,2)\lrar\bbr P^n$ sending
  $[x,y]$ to the unique line $L_{[x,y]}$ passing through the origin and
  parallel to the non-zero vector $x-y$.  The preimage $\pi^{-1}(L_{[x,y]})$
  consists of all pairs $[a,b]$ such that $a-b$ is a multiple of $x-y$.
This can be identified with an ``open'' hemisphere determined   
by the hyperplane orthogonal to $L_{[x,y]}$ (ie $B(S^n,2)$
can be identified with the dual tautological bundle over $\bbr P^n$).
\end{proof}

\begin{exam}\label{c2rn} Similarly we can see that
$B(\bbr^{n+1},2)\simeq\bbr P^{n}$ and that
$B(S^n,2)\hookrightarrow B(\bbr^{n+1},2)$ is a deformation retract.
Alternatively one can see directly that
$B(S^n,2)\simeq\bbr P^n$ for there are an inclusion $i$ and a retract $r$:
\begin{eqnarray*}
i \co  S^n\hookrightarrow F(S^n,2)\ \ &,&\ \ \  r \co  F(S^n,2)\lrar S^n\cr
x\longmapsto (x,-x)\ \ &&\ \ \ \ \ \ \  (x,y)\mapsto {x-y\over |x-y|}
\end{eqnarray*}
Identify $S^n$ with $i(S^n)$ as a subset of $F(S^n,2)$.
Then $F(S^n,2)$ deformation retracts onto this subset via
 $$
 f_t(x,y) = \left({x-ty\over |x-ty|} , {y-tx\over |y-tx|}\right)
 $$
 (which one checks is well-defined). We have that $f_t$ is
 $\bbz_2$--equivariant with respect to the involution $(x,y)\mapsto
 (y,x)$, that $f_0=id$ and that $f_1 \co  F(S^n,2)\lrar S^n$ is
 $\bbz_2$--equivariant with respect to the antipodal action on $S^n$.  That
 is $S^n$ is a $\bbz_2$--equivariant deformation retraction of $F(S^n,2)$ which
 yields the claim.
\end{exam}

\begin{exam} $B(\bbr^2,3)$ is up to homotopy the
  complement of the trefoil knot in $S^3$.  \end{exam}

\begin{exam}\label{c2rp2} 
There is a projection $B(\bbr P^2,2)\lrar\bbr P^2$
which, to any two distinct lines through the origin in $\bbr^3$, associates
the plane they generate and this is an element of the Grassmann manifold
$Gr_2(\bbr^3)\cong Gr_1(\bbr^3)
= \bbr P^2$. The fiber over a given plane
parameterizes various choices of two distinct lines in
that plane and that is $B(\bbr P^1,2)=B(S^1,2)$.
As we just discussed, this is an open M\"{o}bius band $M$ and
$B(\bbr P^2,2)$ fibers over $\bbr P^2$ with fiber $M$ (see Feder \cite{feder}).
Interestingly $\pi_1(B(\bbr P^2,2))$ is a quaternion
group of order $16$ (Wang \cite{wang}).
\end{exam}

To describe the braid spaces of the circle we can consider the 
multiplication map:
$$m \co  \sp{n}(S^1)\lrar S^1\ \ ,\ \ [x_1,\ldots, x_n]\mapsto
x_1x_2\cdots x_n$$ Morton \cite{morton} shows that $m$ is a locally
trivial bundle with fiber the closed $(n-1)$--dimensional disc and
this bundle is trivial if $n$ is odd and non-orientable if $n$ is
even. In particular $\sp{2}(S^1)$ is the closed M\"{o}bius band.  In
fact one can identify $m^{-1}(1)$ with a closed simplex $\Delta^{n-1}$
so that the configuration space component $m^{-1}(1)\cap B(S^1,n)$
corresponds to the open part. This is a non-trivial construction that
can be found in Morton \cite{morton} and Morava \cite{jack}. Since
$B(S^1,n)$ fits in $\sp{n}(S^1)$ as the open disk bundle one gets that

\begin{prop}\label{cns1} $B(S^1,n)$ is a bundle over $S^1$ with fiber
the open unit disc $D^{n-1}$. This bundle is
trivial if and only if $n$ is odd.
\end{prop}

Examples \ref{c2rn} and \ref{c2rp2} show that when $\dim M$ is odd $\neq 1$
or $M$ is not orientable, then $B(M,k)$ fails to be orientable.
The following explains why this needs to be the case.

\begin{lem}[Folklore]\label{folklore}
  Suppose $M$ is a manifold of dimension $d\geq 2$ and pick $n\geq 2$.  Then
  $B(M,n)$ is orientable if and only if $M$ is orientable of even dimension.
\end{lem}

\begin{proof} 
  We consider the $\mathfrak{S}_n$--covering $\pi \co  F(M,n)\fract{\mathfrak{S}_n}{\lrar}
  B(M,n)$.  If $M$ is not orientable, then so is $M^n$. Now $i \co 
  F(M,n)\hookrightarrow M^n$ is the inclusion of the complement of 
  codimension at least two strata
  so that $\pi_1(F(M,n))\lrar \pi_1(M)^n$ is surjective and hence so is the
  map on $H_1$. The dual map in cohomology is an injection mod $2$ and hence
  $w_1(F(M,n))=i^*(w_1(M^n))\neq 0$ since $w_1(M^n)\neq 0$.  This implies that
  $F(M,n)$ is not orientable if $M$ isn't. It follows that the quotient
  $B(M,n)$ is not orientable as well.

Suppose then that $M$ is orientable.
If $d:= \dim M = 2$, then $M$ is a Riemann surface,
$B(M,n)$ is open in $\sp{n}(M)$ which is a complex manifold and hence is
orientable. Suppose now that $d:= \dim M > 2$ so that
$\pi_1F(M,n) = \pi_1(M^n)$ (since the fat diagonal has codimension $>2$).
Notice that we have an embedding
$\iota \co  B(\bbr^d,n)\hookrightarrow B(M,n)$ coming from the embedding
of an open disc $\bbr^d\hookrightarrow M$. Now $\pi_1(B(\bbr^d,n))=\mathfrak{S}_n$
when $d>2$, and $\iota$ induces a section of the short exact sequence
of fundamental groups for the $\mathfrak{S}_n$--covering $\pi$ so we have
a semi-direct product decomposition
$$\pi_1(B(M,n)) = \pi_1(M^n) \ltimes\mathfrak{S}_n\ , \ \ \ d > 2\ .$$
Let's argue then that $B(\bbr^d,n)$ is orientable if and only if $d$ is
even. Denote by $\tau_x$ the tangent space at $x\in \bbr^d$ and write $\pi \co 
F(\bbr^d,n)\lrar B(\bbr^d,n)$ the quotient map.  A transposition
$\sigma\in\mathfrak{S}_n$ acts on the tangent space to $B(\bbr^d,n)$ at some
chosen basepoint say $[x_1,\ldots, x_n]$ which is identified with the tangent
space $\tau_{x_1}\times\cdots\times\tau_{x_n}$ at say $(x_1,\ldots,
x_n)\in\pi^{-1}([x_1,\ldots, x_n])\subset F(\bbr^d,n)\subset (\bbr^d)^n$.  The
action of $\sigma = (ij)$ interchanges both copies $\tau_{x_i}M$ and
$\tau_{x_j}M\cong\bbr^d$ and thus has determinant $(-1)^d$. Orientation is
preserved only when $d$ is even and the claim follows (for the relation
between orientation and fundamental group see Novikov \cite[Chapter 4]{novikov}).
\end{proof}

Note that the lemma above is no longer true in the one-dimensional
case according to \fullref{cns1}.


\section{Truncated symmetric products and duality}\label{tp}

The heroes here are the truncated symmetric product functors $TP^n$
which were first put to good use by B{\"o}digheimer, Cohen and Milgram
in \cite{bcm} and Milgram and L{\"o}ffler in \cite{lm}. For $n\geq 2$,
define the identification space
$$TP^n(X) := \sp{n}(X)/_{\hbox{\footnotesize$\sim$}}\ ,\ \
[x,x,y_1\ldots, y_{n-2}]\sim [*,*,y_1,\cdots, y_{n-2}]$$
where as always $*\in X$ is the basepoint.
Clearly $TP^1X = X$ and
we set $TP^0(X) = *$. Note that by adjunction of basepoint
$[x_1,\ldots, x_n]\mapsto [*,x_1,\ldots, x_n]$, we obtain topological
embeddings $\sp{n}(X)\lrar\sp{n+1}(X)$ and $TP^n(X)\lrar TP^{n+1}(X)$ of
which limits are $\spy (X)$ and
$TP^{\infty}(X)$ respectively.
We identify $\sp{n-1}(X)$ and $TP^{n-1}(X)$ with their images in $\sp{n}(X)$
and $TP^{n}(X)$ under these embeddings and we write
\begin{equation}\label{tpnbar}
\overline{TP}^n(X):= TP^n(X)/TP^{n-1}(X)
\end{equation}
for the \textit{reduced} truncated product.
These are based spaces by construction. We will set
$\overline{TP}^0(X):= S^0$. The following two properties are crucial.

\begin{thm}\label{property}\
\begin{enumerate}
\item {\rm(Dold and Thom \cite{dt})}\qua
$\pi_i(TP^{\infty}(X))\cong\tilde H_i(X;\bbf_2)$
\item {\rm(Milgram and L{\"o}ffler \cite{lm})}\qua There is a
splitting$$H_*(TP^n(X);\bbf_2)\cong H_*(TP^{n-1}(X);\bbf_2)\oplus
\tilde H_*(\overline{TP}^{n}X;\bbf_2 ).$$
\end{enumerate}
\end{thm}

The splitting in (2) is obtained from the long exact sequence for the
pair $(TP^n(X),$\break$TP^{n-1}(X))$ and the existence of a retract
$H_*(TP^n(X);\bbf_2)\lrar H_*(TP^{n-1}(X);\bbf_2 )$ constructed using
a transfer argument.  In fact this splitting can be viewed as a
consequence of the following homotopy equivalence discussed in
\cite{lm} and Zanos \cite{zanos}.

\begin{lem}
$TP^{\infty}(TP^n(X))\simeq TP^{\infty}(\overline{TP}^n(X))
\times TP^{\infty}(TP^{n-1}(X))$.
\end{lem}

Further interesting splittings of the sort for a variety of other
functors are investigated in \cite{zanos}. The prototypical and basic
example of course is Steenrod's original splitting of the homology of
symmetric products (which holds with integral coefficients).

\begin{thm}[Steenrod, Nakaoka]\label{steenrod}
The induced basepoint adjunction map on
homology $H_*(\sp{n-1}(X);\bbz )\lrar H_*(\sp{n}(X);\bbz )$
is a split monomorphism.
\end{thm}


\subsection{Duality and homological dimension}
The point of view we adopt here is that $B(M,n) = TP^n(M)-TP^{n-2}(M)$
as spaces.  A version of Poincar\'e--Lefshetz duality
(\fullref{duality}) can then be used to relate the cohomology of
$B(M,k)$ to the homology of reduced truncated products. This idea is
of course not so new (see B{\"o}digheimer, Cohen and Taylor \cite{bct}
or M{\`u}i \cite{mui}).

If $U\subset X$ is a closed cofibrant subset of $X$, define in
$\sp{n}(X)$ the ``ideal'':
\begin{equation}\label{ideal}
\underline{U} := \{[x_1,\ldots, x_n]\in\sp{n}(X),
x_i\in U\ \hbox{for some $i$}\}
\end{equation}
For example and if $*\in X$ is the basepoint, then $\underline{*} =
\sp{n-1}(X)\subset\sp{n}(X)$.  Let $S$ be the ``singular set'' in
$\sp{n}(X)$ consisting of unordered tuples with at least two repeated
entries. This is a closed subspace.

\begin{lem}\label{quotient} With $U\neq\emptyset$,
$\sp{n}(X)/({\underline{U}\cup S})= \overline{TP}^{n}(X/U)$.
\end{lem}

\begin{proof} Denote by $*$ the basepoint of $X/U$ which is the image of $U$
  under the quotient $X\lrar X/U$. Then by inspection
$$\sp{n}(X)/(\underline{U}\cup S) = \sp{n}(X/U)/(\underline{*}\cup S)\ .$$
Moding out $\sp{n}(X/U)$ by $S$ we obtain
  $TP^n(X/U)/TP^{n-2}(X/U)$.  Moding out further
  by $\underline{*}$ we obtain the desired quotient.
\end{proof}

The next lemma is the fundamental observation which states that for
$M$ a compact manifold with boundary and $U\hookrightarrow M$ a closed
cofibration, $B(M-U,n) \cong \sp{n}(M) - \underline{U\cup\partial
M}\cup S$ is Poincar\'e--Lefshetz dual to the quotient
$\sp{n}(M)/(\underline{U\cup\partial M}\cup S)$.  More precisely, set
\begin{equation}\label{m+}
\overline{M} = M/(U\cup\partial M)
\end{equation}
with the understanding that $\bar M=M$ if $U\cup\partial M=\emptyset,
\{\hbox{point}\}$.
The following elaborates on \cite[Theorem 3.2]{bcm}.

\begin{lem}\label{duality}
If $M$ is a compact manifold of dimension $d\geq 1$, $U\subset M$ a closed
subset with $M-U$ connected, $U\cap\partial
M=\emptyset$ and $\overline{M}$ as in \eqref{m+}, then
$$ H^i(B(M-U,k);\pm\bbz )\cong \begin{cases}
H_{kd-i}(TP^k(\overline{M}), TP^{k-1}(\overline{M});\bbz ),&\
\hbox{if}\ U\cup\partial M\neq\emptyset,\\ H_{kd-i}(TP^k(M),
TP^{k-2}(M);\bbz ),&\ \hbox{if}\ U\cup\partial M=\emptyset.
\end{cases}
$$
The isomorphism holds with coefficients $\bbf_2$.  When $M$ is even
dimensional and orientable, we can replace $\pm\bbz$ by the trivial
module $\bbz$.  \end{lem}

\begin{proof}
  Suppose $X$ is a compact oriented $d$--manifold with boundary
  $\partial X$.  Then Poincar\'e--Lefshetz duality gives an isomorphism
  $H^{d-q}(X;\bbz )\cong H_q(X,\partial X;\bbz )$.  Apply this to the
  following situation: $X$ is a finite $d$--dimensional CW--complex,
  $V\subset X$ is a closed subset of $X$, and $N$ is a
  tubular neighborhood of $V$ which deformation retracts onto it;
$$V\subset N\subset X$$
$\bar N$ its closure and
$\partial\bar N = \partial (X-N)=\bar N-N$.
Assume that $X-N$ is an orientable $d$--dimensional manifold
with boundary $\partial\bar N$. Then
 we have a series of isomorphisms:
\begin{equation}\label{iso}
H^{d-q}(X-V;\bbz )\cong
H^{d-q}(X-N;\bbz )\cong H_q(X-N,\partial\bar N;\bbz )
\cong H_q(X,V;\bbz )
\end{equation}

Let's now apply \eqref{iso} to the case when $X =
\sp{k}(\overline{M})$ with $M$ as in the lemma and with $V$ the
closed subspace consisting of configurations $[x_1,\ldots, x_k]$ such
that 

(i)\qua $x_i=x_j$ for some $i\neq j$,\qua or 

(ii)\qua for some $i$, $x_i=*$
the point at which $U\cup\partial M$ is collapsed out.  

As discussed in \fullref{quotient},
$\sp{k}(\overline{M})/\underline{*} =
\sp{k}(M)/(\underline{U\cup\partial M})$ so that
$\sp{k}(\overline{M})/V$ $ = \sp{k}(M)/(\underline{U\cup\partial M}\cup
S)$ with $S$ again being the image of the fat diagonal in $\sp{k}(M)$.
Then, according to \fullref{quotient} and to its proof we see that
$$\sp{k}(\overline{M})/V =
\begin{cases} 
TP^k(\overline{M})/TP^{k-1}(\overline{M}),& \hbox
{if $\partial M\neq\emptyset$ or $U\neq\emptyset$}, \\ 
TP^k(M)/TP^{k-2}(M),& \hbox{if $M$ closed and $U=\emptyset$}.
\end{cases}
$$ 
Now $B(M-U,k)\cong \sp{k}(M)- \underline{U\cup\partial M}\cup S
=\sp{k}(\overline{M})-V$ is connected (since $M-U$ is), it is $kd$
dimensional and is orientable if $M$ is even dimensional
orientable (\fullref{folklore}). 
Applying \eqref{iso} yields the
result in the orientable case. When $B(M-U,k)$ is non orientable,
Poincar\'e--Lefshetz duality holds with twisted coefficients.
\end{proof}

A version of this lemma has been greatly exploited in \cite{bcm,ks} to
determine the homology of braid spaces and analogs. The following is
immediate.

\begin{cor}\label{conR} With $M$, $U\subset M$ as in \fullref{duality},
let
$$R_k = \begin{cases}
\conn(TP^k(\overline{M})/TP^{k-1}(\overline{M})), &\hbox{if}\ U\cup\partial M\neq\emptyset,\\
\conn(TP^k(M)/TP^{k-2}(M)), &\hbox{if}\ U\cup\partial M=\emptyset.
\end{cases}
$$
Then $\cohdim_{\pm\bbz}(B(M-U,k)) = dk-R_k-1$.
\end{cor}

\fullref{main3} is now a direct consequence of the following result.

\begin{lem}\label{R}
Let $M, U$ and $\overline{M}$ as above,
$r=\conn(\overline{M})$ with $r\geq 1$. Then
$$R_k \geq \begin{cases}
k+r-1, &\hbox{if}\ U\cup\partial M\neq\emptyset,\\
k+r-2, &\hbox{if}\ U\cup\partial M=\emptyset .
\end{cases}$$
\end{lem}

The proof of this key lemma is based on a computation of Nakaoka
\cite[Proposition 4.3]{nakaoka}.  We write $Y^{(k)}$ for the $k$--fold smash
product of a based space $Y$ and $X_{\mathfrak{S}_k}$ the orbit space of a
$\mathfrak{S}_k$--space $X$.

\begin{thm}[Nakaoka]\label{nakak}
If $Y$ is $r$--connected, then
$(Y^{(k)}/\Delta_{\rm fat})_{\mathfrak{S}_k}$ 
is $r+k-1$--connected. 
\end{thm}

\begin{rem}
  In fact nakaoka only proves the homology version of this result and
  also assumes $r\geq 1$. An inspection of his proof shows that $r\geq
  0$ works as well. Also his homology statement can be upgraded to a
  genuine connectivity statement. To see this, we can assume that
  $k\geq 2$ (the case $k=1$ being trivial).  One needs to show in that
  case that $\pi_1((Y^{(k)}/\Delta_{\rm fat})_{\mathfrak{S}_k})=0$.  This
  follows by an immediate application of Van Kampen and the fact that
  $\pi_1(Y^{(k)}/\mathfrak{S}_k)=\pi_1(\bsp{k}Y)=0$ for $k\geq 2$. To
  see this last statement, recall that the natural map
  $\pi_1(Y)\lrar\pi_1(\sp{k}Y)$ factors through $H_1(Y;\bbz )$ and
  then induces an isomorphism $H_1(Y;\bbz)\cong\pi_1(\sp{k}Y)$ when
  $k\geq 2$ (Smith \cite{smith}). But if $\sp{k-1}(Y)\hookrightarrow
  \sp{k}(Y)$ induces a surjection on fundamental groups, then the
  cofiber is simply connected (Van Kampen).  \end{rem}

\begin{proof} (of \fullref{R} and \fullref{main3}) By construction we
  have the equality $\overline{TP^k}(Y)=
  (Y^{(k)}/\Delta_{\rm fat})_{\mathfrak{S}_k}$.  The connectivity of
  $TP^k(M)/TP^{k-1}(M)$ is (at least) $k+r-1$ according to \fullref{nakak}, while that of $TP^{k-1}(M)/TP^{k-2}(M)$ is at least $k+r-2$
  which means that $\conn (TP^k(M)/TP^{k-2}(M)) \geq k+r-2$ (by the long exact
  sequence of the triple $(TP^{k-2}(M),TP^{k-1}(M),TP^k(M))$). This produces
  the lower bounds on $R_k$ in \fullref{R}.  Since the cohomology of $B(M-U,k)$
  starts to vanish at $dk-R_k$ (\fullref{conR}), \fullref{main3}
  follows.
\end{proof}


\section{Braid spaces of punctured manifolds}\label{punctured}

We start with a simple proof of \fullref{main}, \eqref{main11};\
$\dim M=d\geq 2$ throughout.

\begin{proof}[Proof of \fullref{main}, \eqref{main11}]
This is a direct computation (with $M$ closed)
\begin{eqnarray*}
 H^j(B(M,n);\bbf_2)
&\cong& H_{nd-j}(TP^nM, TP^{n-2}M;\bbf_2 )\ \ \ \ \ \ \ \ \ \ \ \ \ \ \ \ \ \
\ \ \ \ \ \
\hbox{(\fullref{duality})}\\
&\cong&
\tilde H_{nd-j}(\overline{TP}^nM;\bbf_2 )\oplus
\tilde H_{nd-j}(\overline{TP}^{n-1}M;\bbf_2 ) \ \ \ \ \ \ (\text{by}\ \ref{property}, (2))\\
&\cong&H^j(B(M-\{p\},n);\bbf_2)\oplus
H^{j-d}(B(M-\{p\},n-1);\bbf_2 )
\end{eqnarray*}
In this last step we have rewritten
$H_{nd-j}$ as
$H_{(n-1)d-(j-d)}$ and reapplied
\fullref{duality}.
\end{proof}

\textbf{Example}\qua When $M=S^d$ and $n=2$, then
$B(S^d,2)\simeq\bbr P^d$ and $B(S^d-p,2)=B(\bbr^d,2) = \bbr P^{d-1}$
in full agreement with the splitting. This shows more importantly that
the splitting is not valid for coefficients other than $\bbf_2$.  The
general case is covered by the following observation of Segal and
McDuff.

\begin{lem}\label{longexact}{\rm (McDuff \cite{dusa})}\qua There 
is a long exact sequence:
\begin{align*}
\lrar H_{*-d+1}(B(M-*,n-1))&\lrar
H_*(B(M-*,n))\\&\lrar H_*(B(M,n))\lrar
H_{*-d}(B(M-*,n-1))\cdots
\end{align*}
\end{lem}

\begin{proof}
  Let $U$ be an open disc in $M$ of radius $<\epsilon$ and let $N =
  M-U$. We have that $B(M-*,n)\simeq B(N,n)$. There is an obvious
  inclusion $B(N,n)\lrar B(M,n)$ and so we are done if we can show
  that the cofiber of this map is $\Sigma^dB(N,n-1)_+$.  To that end
  using a trick as in \cite{dusa} (proof of theorem 1.1) we replace
  $B(M,n)$ by the homotopy equivalent model $B'(M,n)$ of
  configurations $[x_1,\ldots, x_n]\in B(M,n)$ such that at most one
  of the $x_i$'s is in $U$. The cofiber of $B(N,n)\hookrightarrow
  B'(M,n)$ is a based space at $*$ and consists of pairs $(x,D)\in
  \bar U\times B(N,n-1)$ such that if $x\in\partial\bar U$ then
  everything is collapsed out to $*$. But $U\cong D^d$ and $\bar
  U/\partial \bar U = S^d$ so that the cofiber is the half-smash
  product $S^d\rtimes B(N,n-1) = \Sigma^dB(N,n-1)_+$ as asserted.
\end{proof}

In order to prove \fullref{main} we need the following result of Mostovoy.

\begin{lem}{\rm(Mostovoy \cite{mostovoy})}\label{circle}\qua
There is a homeomorphism $\tp{n}(S^1)\cong\bbr P^n$.\end{lem}

\begin{rem} We only need that the spaces be homotopy equivalent.
It is actually not hard to see that
$\tp{n}(S^1)$ has the same homology as $\bbr P^n$ since it can be
decomposed into cells one for each dimension less than $n$ and with
the right boundary maps. The $k^{\rm th}$ skeleton is $\tp{k}(S^1)$.  Indeed
identify $S^1$ with $[0,1]/\sim$. A point in $\tp{k}(S^1)$ can be
written as a tuple $0\leq t_1\leq\cdots\leq t_k\leq 1$ with
identifications at $t_1=0, t_k=1$ and $t_i=t_{i+1}$. The set of all
such points is therefore the image $\sigma^k$ of a $k$--simplex
$\Delta^k\lrar\tp{k}(S^1)$ with identifications along the faces
$F_i\Delta^k$.  Since all faces corresponding to $t_i=t_{i+1}$ map to
the lower skeleton ($\tp{k-2}(S^1))$ and since the last face
$F_k\Delta^k$ (when $t_k=1$) is identified with the zeroth face
($t_1=0$) in $\tp{k}(S^1)$, the corresponding \textit{chain} map sends
the boundary chain $\partial\sigma_k$ to the image of
$\partial\Delta^k = \sum_{i=0}^k (-1)^iF_i\Delta^k$; that is to the
image of $F_0\Delta^k + (-1)^kF_k\Delta^k$ which is
$(1+(-1)^k)\sigma^{k-1}$.
\end{rem}

We need one more lemma.

\begin{lem} Set $\overline{TP}^0(X) = S^0$. Then
$\overline{TP}^n(X\vee Y) = \bigvee_{r+s=n} \overline{TP}^r(X)
\wedge\overline{TP}^s(Y)$.
\end{lem}

\begin{proof}
  Here the smash products are taken with respect to the canonical
  basepoints of the various $\overline{TP}$'s.  A configuration
  $[z_1,\ldots, z_n]$ in $TP^n(X\vee Y)$ can be decomposed into a pair
  of the form $[x_1,\ldots, x_r]\times [y_1,\ldots, y_s]$ in
  $TP^r(X)\times TP^s(Y)$ for some $r+s = n$.  This decomposition is
  unique if we demand that the basepoint (chosen to be the wedgepoint
  $*$) is not contained in the configuration. The ambiguity coming
  from this basepoint is removed when we quotient out $TP^n(X\vee Y)$
  by $\underline{*}=TP^{n-1}(X\vee Y)$, and when we quotient out
  $\bigcup_{r+s = n}TP^r(X)\times TP^s(Y)$ by those pairs of
  configurations with the basepoint in either one of them.  The proof
  follows.
\end{proof}

We are now in a position to prove the second splitting \eqref{main22}.

\begin{proof}[Proof of \fullref{main}, \eqref{main22}] Let $Q_k = \{p_1,\ldots,
  p_k\}$ be a finite subset of $M$ of cardinality $k$. We note that
  the quotient $M/Q_k$ is of the homotopy type of the bouquet
  $M\vee\underbrace{S^1\vee\cdots\vee S^1}_{k-1}$, and that
  $\overline{TP}^l(S^1) = \bbr P^l/\bbr P^{l-1} = S^l$.  Using field
  coefficients we then have the folowing, where, whenever we quote \fullref{duality},
  we assume that either $M$ is even dimensional orientable or
  that $\bbf=\bbf_2$:
\begin{eqnarray*}
H^j(B&(M&-Q_k,n);\bbf )\\
&\cong& \tilde H_{nd-j}(\overline{TP}^n (M/Q_k))\ \ \ \ \ \ \ \ \ \ \ \ \ \ \ \ \
\hbox{(\fullref{duality} with $U\cup\partial M=Q_k$)}\\
&\cong&
 \tilde H_{nd-j}(\overline{TP}^n(M\vee\bigvee_{k-1}S^1))\\
&\cong& \tilde H_{nd-j}\left(\bigvee_{r+s_1+\cdots + s_{k-1}=n}
\overline{TP}^r(M)\wedge \overline{TP}^{s_1}(S^1)\wedge\cdots\wedge
\overline{TP}^{s_{k-1}}(S^1)\right)\\
&\cong&
 \tilde H_{nd-j}\left(\bigvee_{r+s_1+\cdots + s_{k-1}=n}
S^{n-r}\wedge \overline{TP}^rM \right)\\
&\cong&
\bigoplus_{r+s_1+\cdots + s_{k-1}=n}
 \tilde H_{nd-j-n+r}(\overline{TP}^rM )\\
&\cong&\bigoplus_r
 \tilde H_{nd-j-n+r} (\overline{TP}^rM )^{\oplus p(k-1,n-r)}\\
&\cong&\bigoplus_{r=0}^n
H^{j - (n-r)(d-1)}
(B(M-\{p\},r);\bbf )^{\oplus p(k-1,n-r)}\ \ \ \ \ \ \ \ \ \hbox{(\fullref{duality})}
\end{eqnarray*}
This is what we wanted to prove.
\end{proof}


\section{Connectivity of symmetric products}\label{bounds}

In this section we prove \fullref{connectivity} and \fullref{conntwo} of the introduction.

\begin{thm}\label{connectivity2} Suppose $X$ is a based $r$--connected 
  simplicial complex with $r\geq 1$ and let $n\geq 1$. Then
  $\bsp{n}(X)$ is $2n+r-2$--connected.
\end{thm}

\begin{proof} The claim is tautological for $n=1$ and so we assume throughout
  that $n>1$.  We use some key ideas from Arone and Dwyer \cite{dwyer} and
  Kallel and Karoui \cite{kk}.  Start with $X$ simply connected and
  choose a CW complex $Y$ such that $H_*(\Sigma Y)=H_*(X)$. If $X$ is
  based and $r$--connected, then $Y$ is based and $(r-1)$--connected. A
  crucial theorem of Dold \cite{dold} now asserts that $H_*(\sp{n}X)$,
  and hence $H_*(\bsp{n}X)$, only depends on $H_*(X)$ so that in our
  case $H_*(\bsp{n}X)=H_*(\bsp{n}\Sigma Y)$. As before we write
  $X^{(n)}$ the $n$--fold smash product of $X$ so that we can identify
  $\bsp{n}X$ with the quotient $X^{(n)}/\mathfrak{S}_n$ by the action
  of $\mathfrak{S}_n$.  It will also be convenient to write
  $X^{(n)}_{\mathfrak{S}_n}:=X^{(n)}/\mathfrak{S}_n$.  Note that
  $X^{(n)}$ has a preferred basepoint which is fixed by the action of
  $\mathfrak{S}_n$ (ie the action is \textit{based}).  By
  construction we have equivalences
\begin{equation}\label{bspn}
\bsp{n}(\Sigma Y) = (\Sigma Y)^{(n)}_{\mathfrak{S}_n}
= (S^1\wedge Y)^{(n)}_{\mathfrak{S}_n} = (S^1)^{(n)}\wedge_{\mathfrak{S}_n}Y^{(n)}
\end{equation}
where here $A\wedge_{\mathfrak{S}_n}B$ is the notation for the quotient by
the diagonal action of $\mathfrak{S}_n$ on $A\wedge B$ where $A$ admits
a based right action of $\mathfrak{S}_n$ and $B$ a based left action.

We next observe that the quotient $(S^1)^{(n)}/K$ is contractible for
any non-trivial Young subgroup $K=\mathfrak{S}_{k_1}\times
\mathfrak{S}_{k_2}\times\cdots\times\mathfrak{S}_{k_r}\subset
\mathfrak{S}_n$, $\sum k_i=n$.  
This follows from the fact that $(S^1)^{(n)}/K =
S^n/K=S^{k_1}/{\mathfrak S_{k_1}}\wedge \cdots\wedge
S^{k_r}/{\mathfrak S_{k_r}}$, and that for some $k_i\geq 2$,
$S^{k_i}/{\mathfrak S_{k_i}}=\bsp{k_i}(S^1)$ is contractible since
the basepoint inclusion $\sp{k_i-1}(S^1)\lrar\sp{k_i}(S^1)$ is a
homotopy equivalence between two copies of the circle (see section
\ref{braids}).
We can then use \cite[Proposition 7.11]{dwyer} to
conclude that $(S^1)^{(n)}\wedge_{\mathfrak{S}_n} \Delta_{\rm fat}$ is
contractible with $\Delta_{\rm fat}$ as in \eqref{fat}.  This subspace can
then be collapsed out in the expression of $\bsp{n}(\Sigma Y)$ of
\eqref{bspn} without changing the homotopy type and one obtains
\begin{equation}\label{finalform}
\bsp{n}(\Sigma Y)\simeq (S^1)^{(n)}\wedge_{\mathfrak{S}_n}
\left(Y^{(n)}/\Delta_{\rm fat}\right)\ .
\end{equation}
The point of expressing $\bsp{n}(X)$ in this form is to take advantage
of the fact that the action of $\mathfrak{S}_n$ on $Y^{(n)}/\Delta_{\rm fat}$ is
based free (ie, free everywhere but at a single fixed point say $x_0$
to which the entire $\Delta_{\rm fat}$ is collapsed out).

Consider the projection $W_n :=
S^{n}\times_{\mathfrak{S}_n}(Y^{(n)}/\Delta_{\rm fat}) \rightarrow
(Y^{(n)}/\Delta_{\rm fat})_{\mathfrak{S}_n}$. This map is a fibration on the
complement of the point $x_0$ with fiber $S^n$ there, and over $x_0$ the fiber
is $F_0 = S^{n}/\mathfrak{S}_n$ (which is contractible).  The space
$\bsp{n}(\Sigma Y)$ in \eqref{finalform} is obtained from $W_n$ by collapsing
out $F_0$ (being contractible this won't matter) and
$X_n:=*\times_{\mathfrak{S}_n}(Y^{(n)}/\Delta_{\rm fat}) =
(Y^{(n)}/\Delta_{\rm fat})_{\mathfrak{S}_n}$.  Consider the sequence of maps
$(S^n,*)\lrar (W_n,X_n)\fract{}{\lrar} (X_n,X_n)$. This is a fibration away from the
point $x_0\in X$ as we pointed out.  One can then construct a relative serre
spectral sequence (as in \cite[Section 6]{kk}) with $E^2$--term:
$$E^2 = \tilde H_*(X_n; \tilde H_*(S^n))\ \Longrightarrow\
H_*(W_n,X_n)\cong H_*(\bsp{n}(\Sigma Y))$$ 
But $X_n$ is $r+n-2$--connected (\fullref{nakak}), $r+n-2\geq 1$, so that
the $E^2$--term is made out of terms of homological dimension $r+n-1+n =2n+r-1$
or higher which implies that $\bsp{n}(\Sigma Y)=\bsp{n}(X)$ 
has trivial homology up to $2n+r-2$. 
But $\bsp{n}(X)$ is simply connected if $n\geq 2$ (see remark after
\fullref{nakak}) and the proof follows by the Hurewicz Theorem.
\end{proof}

\begin{exam} There is a homotopy equivalence
\ $\bsp{2}(S^k)\simeq \Sigma^{k+1}\bbr P^{k-1}$\ 
(see Hatcher \cite[Chapter 4 , Example 4K.5]{hatcher}). This
space is $k+1 = 4+ (k-1)-2$--connected as predicted and this is sharp. 
 \end{exam}


\subsection{Two dimensional complexes}\label{twodim}
To prove \fullref{conntwo} we use a minimal and explicit complex
constructed in \cite{ks2}. The existence of this complex is due to the simple
but exceptional property in dimension two that $\sp{n}{D}$, where
$D\subset\bbr^2$ is a disc, is again a disc of dimension $2n$.  Write $X =
\bigvee^wS^1\cup (D^2_1\cup\cdots\cup D^2_r)$ and denote by $\star $ the
symmetric product at the chain level.  In \cite{ks2} we constructed a space
$\tsp{n}X$ homotopy equivalent to $\sp{n}(X)$ and such that $\tsp{}X\simeq
\coprod_{n\geq 0}\tsp{n}X$ has a multiplicative cellular chain complex
generated under $\star $ by a zero dimensional class $v_0$, degree one classes
$e_1,\ldots, e_w$ and degree $2s$ classes $\sp{s}D_i$, $1\leq i\leq r$, $1\leq
s$, under the relations
\begin{eqnarray*}
e_i\star e_j = -e_j\star e_i \,\ (i \neq j)\ ,\ \ e_i\star e_i = 0 \ ,\ \\
\sp{s}D_i\star \sp{t}D_i = {s+t\choose t}\sp{s+t}D_i \ .
\end{eqnarray*}
The cellular boundaries on these cells were also explicitly computed
(but we don't need them here). The point however is that
a cellular chain complex for $\tsp{n}(X)$ consists of the subcomplex
generated by cells
$$v_0^r\star e_{i_1}\star \cdots \star e_{i_t}\star
\sp{s_1}(D_{j_1})\star \cdots \star \sp{s_l}(D_{j_l})$$
with $r+t+s_1+\cdots +s_l= n$ and $t\leq w$ where $w$ again is the number
of leaves in the bouquet of circles. The dimension of such a cell is
$t+2(s_1+\cdots + s_l)$ for pairwise distinct indices among
the $e_i$'s.

A reduced cellular complex for $\bsp{n}X$ can then be taken to be
the quotient of $C_*(\tsp{n}X)$ by the summand
$v_0C_*(\tsp{n-1}X)$ and this has cells of the form
$$e_{i_1}\star \cdots \star e_{i_t}\star \sp{s_1}(D_{j_1})\star \cdots \star \sp{s_l}(D_{j_l})$$
with $t+s_1+\cdots +s_l= n$. The dimension of such a cell is $t+2(s_1+\cdots +
s_l)=2n-t$.  The smallest such dimension is $2n-\min(w,n)$. This means that
$\conn(\tsp{n}X/\tsp{n-1}X) = \conn(\bsp{n}X) \geq 2n-\min(w,n)-1$ and
\fullref{conntwo} follows.

\begin{exam}
  A good example to illustrate \fullref{conntwo} is when $S$ is a
  closed Riemann surface of genus $g$. It is well-known that for $n\geq 2g-1$,
  $\sp{n}(S)$ is an analytic fiber bundle over the Jacobian (by a result of
  Mattuck)
$$\bbp^{n-g}\lrar\sp{n}(S)\fract{\mu}{\lrar} J(S)$$
where $\mu$ is the Abel--Jacobi map. In fact this is the projectivisation of an
$n-g+1$ complex vector bundle over $J(S)$. Collapsing out fiberwise the hyperplanes
$\bbp^{n-g-1}\subset\bbp^{n-g}$ we get a fibration $\zeta_n \co  S^{2n-2g}\lrar
E_n\lrar J(S)$ with a preferred section, so that for $n\geq 2g$, $\bsp{n}(S)$
is the cofiber of this section. This is $2n-2g-1$--connected as predicted,
and in fact $\tilde H_*(\bsp{n}(S))=
\sigma^{2n-2g}H_*(J(S))$ where $\sigma$ is a formal
suspension operator which raises degree by one.
\end{exam}

\subsection{Connectivity and truncated products}\label{spec}
The homology of truncated products, and hence of braid spaces,
is related to the homology of symmetric products via a very useful
spectral sequence introduced in \cite{bcm}.
This spectral sequence has been used and adapted with relative
success to other situations; eg \cite{ks}. The starting point is the
duality in \fullref{duality}. The problem of computing
$H^*(B(M,n);\bbf )$ becomes then one of computing the homology of the
relative groups $H_{*}(TP^n\overline{M}, TP^{n-2}\overline{M};\bbf
)$. The key tool is the following \textit{Eilenberg--Moore type}
spectral sequence with field coefficients $\bbf$.

\begin{thm}\label{specseq}{\rm\cite{bcm}}\qua
  Let $X$ be a connected space with a non-degenerate basepoint. Then there is
  a spectral sequence converging to $H_{*}(TP^{n}(X), TP^{n-1}(X);\bbf )$ ,
  with $E^1$--term
\begin{equation}\label{unpunctured}
  \bigoplus_{i+ 2j= n}
  H_*(\sp{i}X,\sp{i-1}X)\tensor H_*(\sp{j}(\Sigma X),\sp{j-1}(\Sigma X))
\end{equation}
and explicit $d^1$ differentials.
\end{thm}

Field coefficients are used here because this spectral sequence uses
the Kunneth formula to express $E^1$ as in \eqref{unpunctured}. Here
$\sp{-1}(X)=\emptyset$ and $\sp{0}(X)$ is the basepoint.

\begin{exam} When $X=S^1$, then $H_*(TP^{n}(S^1),TP^{n-1}(S^1))=\tilde
  H_*(S^n)$.  Since $\sp{i}S^1\simeq S^1$ for all $i\geq 1$, the spectral
  sequence in this case has $E^1$--term of the form
$$H_*(S^1,*)\tensor H_*(\bbp^{{n-1\over 2}},\bbp^{{n-1\over 2}-1})
= \sigma\tilde H_*(S^{n-1}) = \tilde H_*(S^n)$$
if $n$ is odd (where $\sigma$ is the suspension operator), or
$E^1_{*,*}= H_*(\bbp^{(n/2)},\bbp^{(n/2)-1})=
\tilde H_*(S^n)$ if $n$ is even. In all cases the spectral sequence
collapses at $E^1$.
\end{exam}

Now \fullref{duality} combined with \fullref{specseq} gives an easy
method to produce upper bounds for the non-vanishing degrees of $H^*(B(M,n))$.
The least connectivity of the terms $\bsp{i}X\times\bsp{j}(\Sigma X)$ for
$i+2j=n$ translates by duality to such an upper bound.  This was in fact
originally our approach to the cohomological dimension of braid spaces.  We
illustrate how we can apply this spectral sequence by deriving \fullref{twocomplexes} from \fullref{conntwo}.

\begin{proof}[Proof of \fullref{twocomplexes}]
Suppose $Q\cup\partial S\neq\emptyset$.
The spectral sequence of \fullref{specseq} converging to
the homology of $(TP^k(\overline{S}),TP^{k-1}(\overline{S}))$ 
takes the form
\begin{equation}\label{e1}
  E^1 = \tilde H_*(\bsp{k}\overline{S})
  \bigoplus\oplus_{i+2j=k}(H_*(\bsp{i}\overline{S})\otimes 
  H_*(\bsp{j}(\Sigma \overline{S}))
  \bigoplus \tilde H_*(\bsp{k/2}(\Sigma \overline{S}))
\end{equation}
(if $k$ odd, the far right term is not there).
We have that $R_k$ (as in \fullref{conR})
is at least the connectivity of this $E^1$--term.
Since $\overline{S}$ is a two dimensional complex,
the connectivity of $\bsp{i}(\overline{S})$
is at least $2i-\min(w,i)-1$ (for some $w\geq 0$).
The connectivity of $\bsp{j}(\Sigma \overline{S})$ is at least
$2j + r-2\geq 2j-1$
since $\Sigma \overline{S}$ is now simply connected
(\fullref{connectivity2}).
The connectivity of
$\bsp{i}(\overline{S})\wedge \bsp{j}(\Sigma \overline{S})$
for non-zero $i$ and $j$ is then at least
$$(2i-\min(w,i)-1)+(2j-1)+1 = i+k-\min(w,i)-1$$
When $i=0$, then $j={k\over 2}$ ($k$ even) and
$\conn(\bsp{k/2}(\Sigma \overline{S}))\geq k-1$.
The connectivity of the $E^1$--term \eqref{e1} is at least the minimum of
$$
\begin{cases}
i+k-\min(w,i)-1,  &  1\leq i\leq k-1,\\
2k-\min(w,k)-1,& i=k,\\
k-1,& i=0.
\end{cases}
$$
which is $k-1$. By duality $H^*(B(S-Q,k)) = 0$ for $* \geq
2k-k+1 = k+1$. If $S$ is closed, then the same argument shows that this bound
needs to be raised by one.
\end{proof}


\section{Stability and section spaces}\label{stability}

In this final section,
we extrapolate on standard material and make slightly more precise a
well-known relationship between configuration spaces and section spaces
\cite{dusa,bct,segal1,quarterly}.

When manifolds have a boundary or an end (eg a puncture), 
one can construct embeddings
\begin{equation}\label{marching}
+ \co  B(M,k)\lrar B(M,k+1)\ .
\end{equation}
by ``addition of points'' near the boundary, near ``infinity'' or near the
puncture. In the case when $\partial M\neq\emptyset$ for example,
one can pick a component $A$
of the boundary and construct a nested sequence of collared
neighborhoods $V_1\supset V_2\supset\cdots \supset A$ together with
sequences of points $x_k\in V_{k}-V_{k+1}$. There are then embeddings
$B(M-V_k,k)\lrar B(M-V_{k+1},k+1)$ sending $\sum z_i$ to $\sum
z_i+x_k$. Now we can replace $B(M-V_k,k)$ by $B(M-A,k)$ and then by
$B(M,k)$ up to small homotopy.  In the direct limit of these
embeddings we obtain a space denoted by $B(M,\infty )$.  Note that an
easy analog of Steenrod's splitting \cite{bcm} gives the splitting
\begin{equation}\label{split2}
H_*(B(M,\infty ))\cong\bigoplus_{k=0}
H_*(B(M,k+1), B(M,k))
\end{equation}
(here $B(M,0)=\emptyset$).  In fact \eqref{split2} is a special case
of a trademark \textit{stable splitting} result for configuration
spaces of open manifolds or manifolds with boundary.  Denote by
$D_k(M)$ the cofiber of \eqref{marching}. For example
$D_1(M)=B(M,1)=M$.

\begin{thm}\label{split3}{\rm (B{\"o}digheimer \cite{bodig}, 
Cohen \cite{cohen})}\qua For $M$ a manifold with non-empty boundary,
there is a stable splitting (ie, after sufficiently many suspensions):
$$B(M,k)\simeq_s\bigvee_{i=0}^kD_i(M)$$
\end{thm}

The classical case of $M=D^n$ (closed $n$--ball) is due to Victor
Snaith.  A short and clever argument of proof for this sort of
splittings is due to Fred Cohen \cite{cohen}. The next stability bound is
due to Arnold and a detailed proof is in an appendix of \cite{segal1}.

\begin{thm}{\rm (Arnold)}\label{arnold}\qua The embedding $B(M,k)\hookrightarrow
  B(M,k+1)$ induces a homology monomorphism and a homology equivalence up to
  degree $[k/2]$.
\end{thm}

The monomorphism statement is in fact a consequence of \eqref{split2}.
Arnold's range is not optimal. For instance

\begin{thm}{\rm \cite{ks}}\qua If $S$ is a compact Riemann surface and $S^*=S-\{p\}$, 
then $B(S^*,k)\hookrightarrow B(S^*,k+1)$ is a homology equivalence up
to degree $k-1$.
\end{thm}

We define $s(k)$ to be the homological connectivity
of $+ \co  B(M,k)\lrar B(M,k+1)$ (see \fullref{homstab}) .  By Arnold, $s(k)\geq
[k/2]$ .


\subsection{Section spaces}\label{sectionspace}
If $\zeta \co  E\lrar B$ is a fiber bundle over a base space $B$, we
write $\Gamma (\zeta )$ for its space of sections.  If $\zeta$ is
trivial then evidently $\Gamma(\zeta )$ is the same as maps into the
fiber.  Let $M$ be a closed smooth manifold of dimension $d$,
$U\subset M$ a closed subspace and $\tau^+M$ the fiberwise one-point
compactification of the tangent bundle over $M$ with fiber $S^d = \bbr^d
\cup\{\infty\}$. Then $\tau^+M\lrar
M$ has a preferred section $s_{\infty}$ which is the section at
$\infty$ and we let $\Gamma(\tau^+M;U )$ be those sections which
coincide with $s_{\infty}$ on $U$. Note that $\Gamma (\tau^+M )$
splits into components indexed by the integers as in
$$\Gamma (\tau^+M) := \coprod_{k\in\bbz} \Gamma_k(\tau^+M)\ .$$ 
This degree arises as follows. Let $s \co  M\lrar\tau^+M$ be a section.
By general position argument it intersects $s_{\infty}$ at a finite
number of points and there is a sign associated to each point.  This
sign is defined whether the manifold is oriented or not (as in the
definition of the Euler number). The degree is then the signed sum.
Similarly we can define a (relative) degree of sections in $\Gamma
(\tau^+M;U)$.

Observe that if $\tau^+M$ is trivial, then
$\Phi \co  \Gamma (\tau^+M)\fract{\simeq}{\lrar}\Map(M,S^d)$, where
$d=\dim M$.
The components of $\Map(M,S^d)$ are indexed by the degree of maps
(Hopf), but at the level of components we have the equivalence
$$\Gamma_{k} (\tau^+M)\simeq\map{k+\ell}(M,S^d)$$ where $\ell$ is such
that $\Phi (s_{\infty})\in\map{\ell}$.  In the case when $M=S^{even}$,
then $\Phi (s_{\infty})$ is the antipodal map which has degree $\ell =
-1$ \cite{paolo}. When $M=S$ is a compact Riemann surface, $\ell =-1$
when the genus is even and $\ell = 0$ when the genus is odd \cite{ks}.
Further relevant homotopy theoretic properties of section spaces are
summarized in the appendix.

\subsection{Scanning and stability}\label{scan}
A beautiful and important connection between braid spaces and section
spaces can be found for example in \cite{segal2,dusa,quarterly} (see
Crabb and James \cite{james} for the fiberwise version).  This
connection is embodied in the ``scanning'' map
\begin{equation}\label{scanning}
S_k \co  B(M-U,k)\lrar \Gamma_k(\tau^+M; U\cup\partial M )
\end{equation}
where $U$ is a closed subspace of $M$. Here and throughout we assume
that removing a subspace as in $M-U$ doesn't disconnect the space. The
scanning map has very useful homological properties. A sketch of the
construction of $S_k$ for closed Riemaniann $M$ goes as follows (for a
construction that works for topological manifolds see for example
Dwyer, Weiss and Williams \cite{dww}).  First construct $S_1\co 
M-U\lrar\Gamma_1$. We can suppose that $M$ has a Riemannian metric and
use the existence of an exponential map for $\tau M$ which is a
continuous family of embeddings $\exp_x\co  \tau_xM\lrar M$ for $x\in M$
such that $x\in \im(\exp_x)$ and $\im(\exp_x)^+\cong\tau_x^+M$ (the fiber
at $x$ of $\tau^+M$).  By collapsing out for each $x$ the complement
of $\im(\exp_x)$ we get a map $c_x \co  M\lrar \im(\exp_x)^+\cong\tau_x^+M$
Let $V$ be an open neighborhood of $U$, $M-V\lrar M-U$ being a
deformation retract. Then we have the map
$$S_1 \co  M-V\lrar\Gamma (\tau^+M)\ ,\ y\mapsto (x\mapsto c_x(y))\in
\tau_x^+M\ .$$ Observe that for $x$ near $U$, the section $S_1(y)$ agrees
with the section at infinity (ie, we say it is \textit{null}). In fact and more
precisely, $S_1$ maps into $\Gamma^c(\tau^+M,U)$ the space of sections
which are null outside a compact subspace of $M-U$.  A deformation
argument shows that $\Gamma^c\simeq\Gamma$.  It will be convenient to
say that a section $s\in\Gamma$ is \textit{supported} in a subset
$N\subset M$ if $s=s_{\infty}$ outside of $N$. A useful observation is
that if $s_1,s_2$ are two sections supported in closed $A$ and $B$ and
$A\cap B=\emptyset$, then we can define a new section which is
supported in $A\cup B$, restricting to $s_1$ on $A$ and to $s_2$ on
$B$.

Extending $S_1$ to $S_k$ is now easy.  We first choose
$\epsilon > 0$ so that $B^{\epsilon}(M,k)$ the closed subset of
$B(M,k)$ where particles have pairwise separation $\geq 2\epsilon$ is
homotopic to $B(M,k)$ (this is verified in \cite[Lemma 2.3]{dusa}).
We next choose the exponential maps to be supported in neighborhoods
of radius $\epsilon$.  Given a finite subset $Q:=\{y_1,\ldots, y_k\}\in
B^{\epsilon}(M-U,k)$, each point $y_i$
determines a section supported in $V_i := \im(\exp_{y_i} )$.  Since the
$V_i$'s are pairwise disjoint, these sections fit together to give a
section $s_Q$ supported in $\bigcup V_i$ so that $S_k(Q):= s_Q$.

When $M$ is compact with boundary, then we get the map in \eqref{scanning} by
replacing $B(M-U,k)$ by $B(M-U\cup\partial M,k)$ and
$\Gamma^c(\tau^+M,U)$ by $\Gamma (\tau^+M, U\cup\partial M)$ the space
of sections that are null outside a compact subspace of
$M-U\cup\partial M$. We let $s(k)$ be the stability range of the
map $B(M-U,k)\lrar B(M-U,k+1)$ (as in \S6.1)

The next proposition is a follow up on a main result of \cite{dusa} (see
also \cite{quarterly}). 

\begin{prop}\label{dusa1} Suppose $M$ is a closed manifold and $U\subset M$
a non-empty closed subset, $M-U$ connected. 
Then the map $S_{k*}\co  H_*(B(M-U,k))\lrar 
H_*(\Gamma_k(\tau^+M,U))$
is a monomorphism in all dimensions and an isomorphism up to dimension $s(k)$.
\end{prop}

\begin{proof} It is easy to see that the maps $S_k$ for various $k$ are
  compatible up to homotopy with stabilization so we obtain a map $S \co 
  B(M,\infty )\lrar \Gamma_{\infty}(\tau^+M,U):=\lim_k\Gamma_k(\tau^+M,U)$ 
which according to the main
  theorem of McDuff is a homology equivalence (in fact all components of
  $\Gamma (\tau^+M,U)$ are equivalent and $\Gamma_{\infty}$ can be chosen to
  be the component containing $s_{\infty}$).  But according to \eqref{split2}
  $H_*(B(M-U,k))\rightarrow H_*(B(M-U,\infty ))$ is a monomorphism, and then an
  isomorphism up to dimension $s(k)$. The claim follows.
\end{proof}

This now also implies our last main result from the introduction.

\begin{proof}[Proof of \fullref{main4}] Suppose that $M$ is a closed
manifold of dimension $d$, $U$ a small open neighborhood of the basepoint $*$
 and consider the fibration (see \hyperlink{App}{the appendix})
$$\Gamma_k (\tau^+M;\bar U)\lrar\Gamma_k(\tau^+M)\lrar S^d$$ The main
point is to use the fact as in \cite[proof of Theorem 1.1]{dusa} that
scanning sends the exact sequence in \fullref{longexact} to the Wang
sequence of this fibration.  Let $N=M-U$ so that we can identify
$\Gamma_k (\tau^+M;\bar U)$ with $\Gamma_k(\tau^+N;\partial N)$ which
we write for simplicity $\Gamma^c_k(\tau^+N)$ as before.  Under these
identifications and by a routine check we see that scanning induces
commutative diagrams:
$$\small
\begin{matrix}
\!\!\rightarrow\!\!&H_{q-d+1}(B(N,k-1))&\!\!\rightarrow\!\!&
H_q(B(N,k))&\!\!\rightarrow\!\!&H_q(B(M,k))&\!\!\rightarrow\!\!&H_{q-d}(B(N,k-1))
&\!\!\rightarrow\!\!\\
&\decdnar{S}&&\decdnar{S}&&\decdnar{S}&&\decdnar{S}&\\
\!\!\rightarrow\!\!&H_{q-d+1}(\Gamma^c_{k}(\tau^+N))&\!\!\rightarrow\!\!&
H_q(\Gamma^c_{k}(\tau^+N))&\!\!\rightarrow\!\!&H_q(\Gamma_{k}(\tau^+M))
&\!\!\rightarrow\!\!&H_{q-d}(\Gamma^c_{k}(\tau^+N))
&\!\!\rightarrow\!\!
\end{matrix}
$$
where the top sequence is the homology exact sequence for the pair
$(B(M,k),B(N,k))$ as discussed in \fullref{longexact}
and the lower exact sequence is the Wang sequence of
the fibration $\Gamma_k(\tau^+M)\lrar S^d$.
According to \fullref{dusa1}, the map
$S_{k*}\co  H_q(B(N,k))$ $\lrar H_q(\Gamma^c_k(\tau^+N))$ is an isomorphism
up to degree $q=s(k)$. It follows that all vertical maps
in the diagram above involving the subspace $N$ together with the 
next map on the right (which doesn't appear in the diagram)
are isomorphisms whenever $q\leq s(k-1)\leq s(k)$. By the $5$--lemma the
middle map is then an isomorphism within that range as well. This proves
the proposition.
\end{proof}

We can say a little more when $k=1$, $M$ closed always.

\begin{lem}\label{s1} 
The map $S_1 \co M\lrar \Gamma_{1} (\tau^+M)$
induces a monomorphism in homology in degrees $r+1,r+2$, where
$r=\conn(M)$, $r\geq 1$. 
\end{lem}

\begin{proof}
Consider $\Gamma(s\tau^+M)$ the space of sections
of the fibration
$s\tau^+M\lrar M$ obtained from $\tau^+M$ by applying fiberwise
the functor $\spy$. It is easy to see that scanning has a stable analog
$st\co  \spy (M_+)\lrar \Gamma (s\tau^+M)$ but harder to verify that
$st$ is a (weak) homotopy
equivalence \cite{dww,quarterly}. Note that $\spy (M_+)\simeq\spy
M\times\bbz$
and $\spy (M)$ is equivalent to a connected component (any of them) 
say $\Gamma_0 (s\tau^+M )$.
By construction the following diagram homotopy commutes
$$
\begin{matrix}\label{fromstosp}
M&\fract{S_1}{\lrar}&\Gamma_{1} (\tau^+M)\\
\decdnar{}&&\decdnar{\alpha}\\
\spy (M)&\fract{st}{\lrar}&\Gamma_0 (s\tau^+M )
\end{matrix}
$$
where the right vertical map $\alpha$ is induced from the natural fiber
inclusion $\alpha \co  S^d\hookrightarrow\spy (S^d)$. When $M$ is
$r$--connected, the map $M\lrar \spy (M)$ induces an isomorphism in homology in
dimensions $r+1$ and $r+2$ \cite[Corollary 4.7]{nakaoka}.  This means that
the composite $M\rightarrow \Gamma_{1} (\tau^+M)\rightarrow \Gamma_1 (s\tau^+M
)$ is a homology isomorphism in those dimensions and the claim follows.
\end{proof}

\begin{rem} If $M$ has boundary, then by scanning $M_0:=M-\partial M$ we
  obtain a map into the compactly supported sections $\Gamma
(\tau^+M)$.  This map extends to a map $S \co  M/\partial M\lrar \Gamma
(\tau^+M)$ which is according to Aouina and Klein \cite{aouina}
$(d-r+1)$--connected if $M$ is $r$--connected of dimension $d\geq 2$.
\end{rem}

\hypertarget{App}{\smash{$\phantom{9}$}}
\section{Appendix: Some homotopy properties of section spaces}

All spaces below are assumed connected.
We discuss some pertinent statements from Switzer \cite{switzer}.
Let $p\co  E\lrar B$ be a Serre fibration, $i\co  A\hookrightarrow X$
a cofibration ($A$ can be empty) and  $u\co  X\lrar E$ a given map.
Slightly changing the notation in that paper, we define
$$\Gamma_u (X,A; E,B) = \{f\co  X\lrar E\ |\ f\circ i = u\circ i,
p\circ f = p\circ u\}$$
This is a closed subspace of the space of all maps $\Map(X,E)$ and is in
other words the solution space for the extension problem  
$$\xymatrix{
A\ar[r]^{ui}\ar[d]^i&E\ar[d]^p\\
X\ar[r]_{pu}\ar[ru]^u&B
}
$$
with data $u_{|A}\co  A\lrar E$ and $pu \co  X\lrar B$. When $A=\{x_0\}$ and $B =
\{y_0\}$ then $\Gamma (X,x_0;E,y_0) = \bmap{}(X,E)$ is the space of based maps
from $X$ to $Y$ sending $x_0$ to $y_0$.  On the other hand and when $X=B$ and
$A=\emptyset$, then $\Gamma_u(B,\emptyset;E,B) = \Gamma (E)$ is the section
space of the fibration $\zeta = (E\fract{p}{\lrar} B)$.

\begin{prop}{\rm\cite{switzer}}\label{switzer}\qua
\begin{itemize}
\item
If $A\subset X'\subset X$ is a nested sequence of NDR pairs, and
$j\co  X'\hookrightarrow X$ the inclusion, then the induced map
$\Gamma_u(X,A; E,B)\lrar \Gamma_{uj}(X',A;E,B)$\
yields a fibration with $\Gamma_u (X,X';E,B)$ as fibre.
\item
If $E\lrar E'\lrar B$ are two fibrations and $q\co  E\lrar E'$
the projection, then\break the induced map
$\Gamma_u(X,A; E,B)\lrar \Gamma_{qu}(X,A;E',B)$\ 
is a fibration with\break $\Gamma_u (X,A;E,E')$ as fibre.
\end{itemize}
\end{prop}

The first part of Switzer's result implies that restriction of the
bundle $\zeta \co  E\lrar B$ to $X\subset B$ is a fibration $\Gamma
(\zeta )\lrar \Gamma (\zeta_{|X})$ with fiber the section space
$\Gamma (\zeta, X)$ ie, those sections of $\zeta$ which are
``stationary'' over $X$ (compare \cite[Chapter 1, Section 8]{james}).
An example of relevance is when $\zeta = \tau^+M$ is the fiberwise
one-point compactification and $s_{\infty}$ is the section mapping at
infinity.  Denote by $S^d$ the fiber over $x_0\in M$. If $U$ is a
small open neighborhood of $x_0$, then $\Gamma (\zeta_{|\bar U})\simeq
S^d$ and we have a fibration
\begin{equation}\label{fibration}
\Gamma (\tau^+M,\bar U)\lrar\Gamma (\tau^+M)\fract{res}{\lrar} S^d
\end{equation}
where the fiber consists of those sections which coincide with $s_{\infty}$ on
$U$. So for instance if $M=S^d$, $\Gamma (\tau^+M,\bar U)\simeq\Omega^dS^d$
and the fibration reduces to the evaluation fibration
$\Omega^dS^d\rightarrow \Map(S^d,S^d)\rightarrow S^d$.

Finally and
according to \cite[page 29]{james}, if $E\lrar B$ is
a Hurewicz fibration and $s,t$ are two sections, then $s$ and $t$ are
homotopic if and only if they are section homotopic. We use this to deduce the
following lemma.

\begin{lem}
  \label{mono} Let $\pi \co  E\lrar B$ be a fibration with a preferred section
  $s_{\infty}$ (which we choose as basepoint). Then the inclusion $\Gamma
  (E)\lrar\Map(B,E)$ induces a monomorphism on homotopy groups.
\end{lem}

\begin{proof} We give $\Gamma (E)\subset\Map(B,E)$ the common basepoint
  $s_{\infty}$.  An element of $\pi_i\Gamma (E)$ is the homotopy class of a
  (based) map $\phi \co  S^i\lrar\Gamma (E)$ or equivalently a map $\phi \co 
  S^i\times B\lrar E$ (where $\phi (-,b)\in\pi^{-1}(b)$ and $\phi
  (N,-)=s_\infty(-)$, $N$ the north pole of $S^i$) and the homotopy is through
  similar maps. Write $\Phi$ the image of $\phi$ via the composite
  $S^i\lrar\Gamma (E)\lrar\Map(B,E)$. Now $\Phi$ can be viewed as a section
  of $S^i\times E\lrar S^i\times B$ and a null-homotopy of
  $\Phi$ is a homotopy to $id\times s_{\infty}$. Since this null-homotopy can
  be done fiberwise it is a null-homotopy in $\Gamma (E)$ from $\phi$ to
  $s_{\infty}$.
\end{proof}

\bibliographystyle{gtart}
\bibliography{link}

\end{document}